\documentclass[11pt,reqno]{amsart}

\usepackage[utf8]{inputenc}
\usepackage[T1]{fontenc}
\usepackage{amsmath,amssymb,amsthm,mathtools}

\usepackage[margin=1in]{geometry}
\usepackage{hyperref}
\usepackage{tikz-cd}
\usepackage{enumitem}
\newcommand{\rk}{\operatorname{rk}}

\hypersetup{
  colorlinks=true,
  linkcolor=blue!50!black,
  citecolor=blue!50!black,
  urlcolor=blue!60!black
}

\theoremstyle{plain}
\newtheorem{theorem}{Theorem}[section]
\newtheorem{proposition}[theorem]{Proposition}
\newtheorem{lemma}[theorem]{Lemma}

\theoremstyle{definition}
\newtheorem{definition}[theorem]{Definition}
\newtheorem{remark}[theorem]{Remark}

\newcommand{\E}{\mathcal{E}}
\newcommand{\F}{\mathcal{F}}
\newcommand{\OO}{\mathcal{O}}
\renewcommand{\rk}{\operatorname{rk}}
\renewcommand{\deg}{\operatorname{deg}}

\DeclareMathOperator{\ev}{ev}

\numberwithin{equation}{section}

\title[Semistability of Syzygy Bundles of Ulrich Bundles]{Semistability of Syzygy Bundles Associated to Ulrich Bundles on Projective Varieties of Arbitrary Dimension}

\author{Soham Mondal}
\address{Department of Mathematical Sciences, IISER Berhampur, Odisha, India}
\email{getsoham1@gmail.com}

\date{\today}


\begin{document}

\begin{abstract}
Let $X$ be a smooth irreducible projective variety of dimension $n\ge 3$ over an algebraically closed field of characteristic zero, polarized by a very ample line bundle $\OO_X(1)$. Let $\E$ be an Ulrich bundle on $X$. We prove that there exists an explicitly computable integer $M\gg 0$ such that for every $m\ge M$ the global syzygy bundle $S_{\E(m)}$ is slope semistable with respect to $\OO_X(1)$.  This confirms Conjecture~3.11 of Mir\'o-Roig.
\end{abstract}

\maketitle

\section{Introduction}\label{sec:intro}

The study of syzygy bundles—defined as the kernels of evaluation maps of globally generated vector bundles—has been a central mechanism for understanding the projective geometry and cohomological positivity of algebraic varieties. Initiated in part by the foundational work of Ein and Lazarsfeld \cite{EinLazarsfeld1993}, the stability of these bundles carries deep implications for the geometry of the ambient space.

For curves, the stability of syzygy bundles is governed by the seminal theorem of Butler \cite{Butler1994}: if a globally generated bundle on a curve is semistable and has sufficiently large slope relative to the genus, its syzygy bundle is semistable. Extending this to higher-dimensional varieties introduces severe cohomological obstructions, primarily because moving curves constructed as complete intersections within an $n$-dimensional variety produce restriction kernels whose dimension grows polynomially in the twist.

In this paper we resolve the asymptotic semistability question for syzygy bundles associated to Ulrich bundles. These bundles, characterized by maximal global sections and complete cohomological vanishing \cite{CMP2021,Beauville2000}, possess a rigid algebraic structure that makes them natural candidates for this problem. The key insight is a two-step reduction:

\begin{enumerate}[label=(\arabic*)]
  \item \emph{Surface reduction}: slice $X$ with $n-2$ general hyperplanes to obtain a smooth complete intersection surface $Y$ on which $\E|_Y$ remains Ulrich (and hence semistable).
  \item \emph{Decoupled moving curve}: introduce a fixed geometric offset $c\ge 0$ to decouple the curve degree $k$ from the bundle twist $m$ by setting $k=m-c$, which forces the restriction kernel to have \emph{constant} dimension $N_K$ independent of $m$.
\end{enumerate}

This decoupling, together with the arithmetically Cohen--Macaulay (aCM) property of Ulrich bundles, enables a definitive contradiction against any  destabilizing subsheaf.

\subsection*{Main result}

\begin{theorem}[Main Theorem]\label{thm:main}
Let $X$ be a smooth irreducible projective variety of dimension $n\ge 3$ over an algebraically closed field of characteristic zero, polarized by a very ample line bundle $H=\OO_X(1)$. Let $\E$ be an Ulrich bundle on $X$ of rank $r$. Then there exists an explicitly computable integer $M\gg 0$ such that for every $m\ge M$ the global syzygy bundle $S_{\E(m)}$ is $\mu_H$-semistable.
\end{theorem}

\subsection*{Organization} Section~\ref{sec:prelim} fixes notation and collects preliminary results. Section~\ref{sec:reduction} carries out the surface reduction and constructs the decoupled moving curve. Section~\ref{sec:fundamental} establishes the fundamental exact sequence on the curve. Section~\ref{sec:butler} verifies Butler's conditions. Section~\ref{sec:asymptotics} computes the asymptotic expansions. Section~\ref{sec:proof} executes the global contradiction argument and completes the proof of Theorem~\ref{thm:main}.

\subsection*{Acknowledgements} The author thanks the Department of Mathematical Sciences, IISER Berhampur, for support during this work.

\section{Preliminaries}\label{sec:prelim}

Throughout, $X$ denotes a smooth irreducible projective variety of dimension $n\ge 3$ over an algebraically closed field $\mathbf k$ of characteristic zero, polarized by a very ample line bundle $H=\OO_X(1)$. Set $d:=H^n$.

\begin{definition}[Slope and semistability]\label{def:slope}
For a torsion-free coherent sheaf $\F$ on $X$ of positive rank, its \emph{$H$-slope} is
\[
\mu_H(\F):=\frac{c_1(\F)\cdot H^{n-1}}{\rk(\F)}\in\mathbb Q.
\]
The sheaf $\F$ is \emph{$\mu_H$-semistable} (resp.\ \emph{$\mu_H$-stable}) if for every coherent subsheaf $\mathcal G\subset\F$ with $0<\rk(\mathcal G)<\rk(\F)$ one has $\mu_H(\mathcal G)\le\mu_H(\F)$ (resp.\ $<$).
\end{definition}

\begin{definition}[Ulrich bundle]\label{def:ulrich}
A vector bundle $\E$ on $X$ is an \emph{Ulrich bundle with respect to $H$} if
\[
H^i(X,\E(-j))=0\quad\text{for all }\ i\ge 0\ \text{ and }\ 1\le j\le n.
\] Every Ulrich bundle is arithmetically Cohen--Macaulay (aCM), i.e.\ $H^i(X,\E(t))=0$ for all $0<i<n$ and all $t\in\mathbb Z$.
\end{definition}

\begin{definition}[Syzygy bundle]\label{def:syzygy}
Let $V$ be a globally generated vector bundle on a projective variety $Z$. The \emph{syzygy bundle} $S_V$ is the kernel of the evaluation map:
\[
0\longrightarrow S_V\longrightarrow H^0(Z,V)\otimes\OO_Z\xrightarrow{\ \ev\ }V\longrightarrow 0.
\]
\end{definition}

We record the foundational facts cited from the literature.

\begin{theorem}[{\cite[Thm.~2.3.2]{CMP2021}}]\label{thm:ulrich-ss}
Every Ulrich bundle $\E$ on $(X,H)$ is $\mu_H$-semistable.
\end{theorem}

\begin{theorem}[{\cite[Prop.~3.3.3]{CMP2021}}]\label{thm:ulrich-restrict}
Let $\E$ be an Ulrich bundle on $(X,H)$ and let $H'\in|\OO_X(1)|$ be a general hyperplane. Then $\E|_{H'}$ is an Ulrich bundle on $(H',H|_{H'})$.
\end{theorem}

\begin{theorem}[Butler's theorem; {\cite[Thm.~1.1]{Butler1994}}]\label{thm:butler}
Let $C$ be a smooth projective curve of genus $g\ge 1$ and let $V$ be a vector bundle on $C$ satisfying:
\begin{enumerate}[label=\textnormal{(B\arabic*)}]
\item $V$ is globally generated,
\item $V$ is $\mu_C$-semistable,
\item $\mu_C(V)\ge 2g$.
\end{enumerate}
Then the syzygy bundle $M_V$ of $V$ on $C$ is $\mu_C$-semistable.
\end{theorem}

To establish the main semi-stability result, we first require a strict upper bound on the slopes of subsheaves inside the syzygy bundle.

\begin{lemma}[Strict Negativity Gap] \label{lem:strict_negativity}
Let $X$ be a smooth projective variety of dimension $n$ polarized by a very ample line bundle $H = \mathcal{O}_X(1)$. Let $\mathcal{E}$ be an Ulrich bundle on $X$, and fix an integer $m \ge 0$. If $S_{\mathcal{E}(m)}$ denotes the $m$-th syzygy bundle associated to $\mathcal{E}$, then for any coherent subsheaf $\mathcal{F} \subset S_{\mathcal{E}(m)}$ of positive rank,
\[
    c_1(\mathcal{F}) \cdot H^{n-1} \le -1, \quad \text{and consequently} \quad \mu_H(\mathcal{F}) \le -\frac{1}{\operatorname{rk}(\mathcal{F})}.
\]
\end{lemma}

\begin{proof}
Let $r := \operatorname{rk}(\mathcal{F}) \ge 1$, and set $N := h^0(X, \mathcal{E}(m))$. We proceed in three steps.

\noindent \textbf{Step 1: Vanishing of global sections.} The syzygy bundle is defined via the standard exact sequence
\begin{equation} \label{eq:syz_eval}
    0 \longrightarrow S_{\mathcal{E}(m)} \longrightarrow \mathcal{O}_X^{\oplus N} \xrightarrow{\operatorname{ev}_m} \mathcal{E}(m) \longrightarrow 0.
\end{equation}
Taking the associated long exact sequence in cohomology, the induced map on global sections $H^0(\operatorname{ev}_m)$ is the canonical identity map on $H^0(X, \mathcal{E}(m))$ and is therefore an isomorphism. Consequently, the connecting homomorphism forces $H^0(X, S_{\mathcal{E}(m)}) = 0$. Because $\mathcal{F}$ is a subsheaf of $S_{\mathcal{E}(m)}$, it inherits the vanishing $H^0(X, \mathcal{F}) = 0$.

\noindent \textbf{Step 2: Non-positivity of the slope.} The trivial bundle $\mathcal{O}_X^{\oplus N}$ is $\mu_H$-semistable of slope zero. Composing the inclusions yields a subsheaf injection $\mathcal{F} \hookrightarrow \mathcal{O}_X^{\oplus N}$; the semistability of the ambient bundle immediately forces $\mu_H(\mathcal{F}) \le 0$, which is equivalent to $c_1(\mathcal{F}) \cdot H^{n-1} \le 0$. Because the polarization $H$ is very ample, $H^{n-1}$ is represented by an effective curve class, ensuring the intersection pairing takes integer values. Thus, $c_1(\mathcal{F}) \cdot H^{n-1} \in \mathbb{Z}_{\le 0}$.

\noindent \textbf{Step 3: Ruling out slope zero.} Suppose for contradiction that $c_{1}(\mathcal{F}) \cdot H^{n-1} = 0$, meaning $\mu_{H}(\mathcal{F}) = 0$. Because $\mathcal{F} \subset \mathcal{O}_{X}^{\oplus N}$ is a torsion-free coherent subsheaf of a locally free sheaf, we can pass to its reflexive hull (double dual) $\mathcal{F}^{**} \subset (\mathcal{O}_{X}^{\oplus N})^{**} \cong \mathcal{O}_{X}^{\oplus N}$. Since the singular locus of a coherent sheaf has codimension at least two and $X$ is smooth, the first Chern class is invariant under taking the double dual: $c_{1}(\mathcal{F}^{**}) = c_{1}(\mathcal{F})$, which guarantees $\mu_{H}(\mathcal{F}^{**}) = 0$. 

The ambient trivial bundle $\mathcal{O}_{X}^{\oplus N}$ is a polystable vector bundle of slope zero. Because $\mathcal{F}^{**}$ acts as a saturated reflexive subsheaf of a polystable vector bundle having the exact same maximal slope, it must be a locally split direct summand; therefore, $\mathcal{F}^{**}$ is itself a trivial vector bundle, meaning $\mathcal{F}^{**} \cong \mathcal{O}_{X}^{\oplus r}$. This isomorphism immediately mandates that $h^{0}(X, \mathcal{F}^{**}) = r \ge 1$.

Now, consider the composed injective map of sheaves $\mathcal{F}^{**} \hookrightarrow \mathcal{O}_{X}^{\oplus N} \xrightarrow{ev_m} \mathcal{E}(m)$. Because $\mathcal{F} \subset \ker(ev_m) = S_{\mathcal{E}(m)}$, this composite morphism is identically zero on an open subset $U \subset X$ whose complement has codimension at least two. Since $\mathcal{E}(m)$ is locally free (and hence purely torsion-free), it admits no non-zero subsheaves supported in codimension two or higher. Thus, the global sheaf morphism $\mathcal{F}^{**} \to \mathcal{E}(m)$ is identically zero across all of $X$. By the universal property of kernels, the injection factors strictly through the syzygy bundle: $\mathcal{F}^{**} \hookrightarrow S_{\mathcal{E}(m)}$. 

Taking global sections yields $r = h^0(X, \mathcal{F}^{**}) \le h^0(X, S_{\mathcal{E}(m)}) = 0$, an absolute absurdity since $r \ge 1$. 

We conclude that $c_1(\mathcal{F}) \cdot H^{n-1} \le -1$. Dividing by the rank yields the desired slope bound $\mu_H(\mathcal{F}) \le -1/r$.
\end{proof}

\section{Dimensional reduction and the moving curve}\label{sec:reduction}

\subsection{Reduction to a complete intersection surface}\

\begin{proposition}[Iterated Ulrich restriction]\label{prop:iterated}
 Choose $n-2$ general hyperplanes $H_1,\ldots,H_{n-2}\in|\OO_X(1)|$ and define the smooth irreducible complete intersection surface
\[
Y\ :=\ H_1\cap\cdots\cap H_{n-2},\qquad H_Y:=H|_Y.
\]
Then:
\begin{enumerate}[label=\textnormal{(\roman*)}]
\item $\E|_Y$ is an  Ulrich bundle on $(Y,H_Y)$;
\item $\E|_Y$ is $\mu_{H_Y}$-semistable.
\end{enumerate}
\end{proposition}

\begin{proof}
 Define the descending flag
\[
X=X_0\supset X_1\supset\cdots\supset X_{n-2}=Y,\qquad X_k:=X_{k-1}\cap H_k,
\]
where each $H_k\in|\OO_{X_{k-1}}(1)|$ is chosen general so that $X_k$ is smooth irreducible (Bertini). Setting $\E_0:=\E$ and $\E_k:=\E_{k-1}|_{X_k}$, induction on $k$ via Theorem~\ref{thm:ulrich-restrict} shows that each $\E_k$ is Ulrich on $X_k$; in particular $\E_{n-2}=\E|_Y$ is Ulrich on $Y$.

\emph{(ii)} Every Ulrich bundle is $\mu_H$-semistable (Theorem~\ref{thm:ulrich-ss}).
\end{proof}

Set $d_Y:=H_Y^2=H^n=d$ (since $Y=H^{n-2}$) and $\mu_Y:=\mu_{H_Y}(\E|_Y)$.

\begin{definition}[Geometric offset and Butler margin]\label{def:offset}
Define the integer offset
\[
c\ :=\ \max\!\left(0,\ \left\lfloor\frac{K_Y\cdot H_Y-\mu_Y}{d_Y}\right\rfloor+1\right)\in\mathbb Z_{\ge 0}
\]
and the \emph{Butler margin}
\[
B\ :=\ c\,d_Y+\mu_Y-K_Y\cdot H_Y\ \in\ \mathbb Q.
\]
By construction $B>0$ (and, more precisely, $0<B\le d_Y$).
\end{definition}

\begin{remark}
The strict positivity $B>0$ suffices for everything below; the slope $\mu_Y$ is rational in general, so no \emph{a priori} integer lower bound on $B$ can be expected.
\end{remark}

\begin{remark}[Independence of the numerical invariants on the choice of general surface]
Let \(X\) be a smooth projective variety of dimension \(n \geq 3\) polarized by a very ample line bundle \(H = \mathcal{O}_X(1)\). Let \(E\) be a vector bundle of rank \(r\) on \(X\). Let \(Y \subset X\) be any smooth complete intersection surface obtained by cutting \(X\) with \(n-2\) general members of \(|H|\), and write \(H_Y = H|_Y\).

\begin{enumerate}
\item The self-intersection number \(d_Y := H_Y^2 = H^n = d\) is independent of the choice of \(Y\).

\item The intersection number \(K_Y \cdot H_Y\) is independent of the choice of general \(Y\).

\item The slope \(\mu_Y := \frac{c_1(E|_Y) \cdot H_Y}{r}\) is independent of the choice of general \(Y\).
\end{enumerate}

\begin{proof}
(1) By construction \(Y\) is cut out by divisors of class \(H\), so its degree with respect to \(H_Y\) is always \(H^n\), regardless of which general hyperplanes are chosen.

(2) Since \(Y\) is a smooth complete intersection of \(n-2\) divisors each of class \(H\), the adjunction formula gives
\[
K_Y = \bigl(K_X + (n-2)H\bigr)\big|_Y.
\]
Intersecting with \(H_Y\) on \(Y\) therefore yields the intersection number on \(X\)
\[
K_Y \cdot H_Y = \bigl(K_X + (n-2)H\bigr) \cdot H^{n-1}.
\]
The right-hand side is a purely numerical intersection product on \(X\) and does not depend on the particular choice of the hyperplanes defining \(Y\).

(3) Let \(c_1(E)\) denote the first Chern class of \(E\) in the Chow ring (or numerical equivalence class) of \(X\). Since \(E\) is locally free, the restriction satisfies \(c_1(E|_Y) = c_1(E)|_Y\). We must show that the intersection number on \(Y\)
\[
c_1(E|_Y) \cdot H_Y = c_1(E) \cdot H^{n-1}.
\]
The class of \(Y\) in the Chow ring of \(X\) is \([Y] = H^{n-2}\). Intersecting the restricted class \(c_1(E)|_Y\) with the class of a general hyperplane section \(H_Y\) on \(Y\) corresponds, via the projection formula in intersection theory, to intersecting the class \(c_1(E)\) on \(X\) with the product of the class of \(Y\) and one further factor of \(H\):
\[
c_1(E)|_Y \cdot H_Y \quad = \quad c_1(E) \cdot \bigl([Y] \cdot H\bigr) = c_1(E) \cdot H^{n-2} \cdot H = c_1(E) \cdot H^{n-1}.
\]
(The equality holds because all intersections are proper when the hyperplanes are chosen generally, by Bertini’s theorem.) Consequently
\[
\mu_Y = \frac{c_1(E) \cdot H^{n-1}}{r},
\]
which is a fixed numerical invariant of the triple \((X, H, E)\) and does not depend on the particular general surface \(Y\).
\end{proof}

As an immediate consequence, the geometric offset
\[
c = \max\Bigl\{0,\ \Bigl\lceil\frac{K_Y \cdot H_Y - \mu_Y}{d_Y}\Bigr\rceil + 1\Bigr\}
\]
and the Butler margin \(B = c\, d_Y + \mu_Y - K_Y \cdot H_Y\) are likewise independent of the choice of general complete intersection surface \(Y\).
\end{remark}

\begin{definition}[Global moving curve]\label{def:moving_curve}
For each integer $m\ge c$, set $k:=m-c\ge0$. Let $\mathcal{H}_{k}$ denote the Hilbert scheme parametrizing smooth, irreducible complete intersection curves in $X$ cut out by the global linear system $\big|\mathcal{O}_{X}(1)^{\oplus(n-2)}\oplus\mathcal{O}_{X}(k)\big|$. We select a general member $C_{k}\in\mathcal{H}_{k}$, explicitly cut out as
\[
C_{k} := H_{1}\cap H_{2}\cap\dots\cap H_{n-2}\cap D_{k},
\]
where $H_{1},\dots,H_{n-2}\in|\mathcal{O}_{X}(1)|$ are general hyperplanes and $D_{k}\in|\mathcal{O}_{X}(k)|$ is a general hypersurface. Setting $Y:=H_{1}\cap\dots\cap H_{n-2}$, the subscheme $C_{k}$ is naturally exhibited as a smooth, general divisor in the linear system $|\mathcal{O}_{Y}(k)|$.
\end{definition}

\begin{lemma}[Curve Genus]\label{lem:curve_genus}
The moving curve $C_{k} \subset Y$ is linearly equivalent to $kH_{Y}$, and its arithmetic genus $g_{k}$ satisfies the exact relation
\begin{equation}\label{eq:genus_k}
2g_{k} = k^{2}d + k(K_{Y}\cdot H_{Y}) + 2,
\end{equation}
where $d := H^{n} = H_{Y}^{2}$ is the polarized degree of the ambient variety $X$. In particular, as $k \to \infty$, one has the asymptotic form $g_{k} = \frac{d}{2}k^{2} + O(k)$.
\end{lemma}

\begin{proof}
By construction, $C_{k}$ is exhibited as a smooth general divisor in the basepoint-free linear system $|\mathcal{O}_{Y}(k)|$ on the surface $Y$. The adjunction formula on $Y$ gives
\[
2g_{k} - 2 = C_{k} \cdot (C_{k} + K_{Y}) = (kH_{Y}) \cdot (kH_{Y} + K_{Y}) = k^{2}d + k(K_{Y}\cdot H_{Y}).
\]
Adding $2$ to both sides establishes identity (\ref{eq:genus_k}). Dividing by $2$ yields the leading quadratic behavior $\frac{d}{2}k^{2} + O(k)$.
\end{proof}

\section{THE FUNDAMENTAL EXACT SEQUENCE}\label{sec:fundamental}

With the semistability of the restricted bundles established on the auxiliary complete intersection surface $Y$, we now analyze the syzygy bundles along the 1-dimensional curve slices. Our first objective is to track the exact behavior of global sections under restriction.

\begin{theorem}[Surjectivity and kernel invariance]\label{thm:kernel_invariance}
Let $X$ be polarized by $H=\mathcal{O}_X(1)$, let $\mathcal{E}$ be an Ulrich bundle on $X$, and fix the geometric offset $c \ge 0$. For any twist $m \ge c$, let $C_k \in \mathcal{H}_k$ be the global moving curve defined in Definition \ref{def:moving_curve}, exhibited as a smooth divisor in $|\mathcal{O}_Y(k)|$ of degree $k := m - c$ on the intermediate surface $Y := H_1 \cap \dots \cap H_{n-2}$. Then the restriction homomorphism on global sections
\[
\mathrm{res}_k : H^0\big(Y, \, \mathcal{E}(m)|_Y\big) \longrightarrow H^0\big(C_k, \, \mathcal{E}(m)|_{C_k}\big)
\]
is surjective. Moreover, its kernel is canonically isomorphic to the static vector space
\[
K_{m,c} := H^0\big(Y, \, \mathcal{E}(c)|_Y\big),
\]
whose dimension $N_K := h^0(Y, \mathcal{E}(c)|_Y)$ is strictly independent of $m$. Equivalently, one has the numerical identity
\begin{equation}\label{eq:dim_split}
h^0\big(Y, \, \mathcal{E}(m)|_Y\big) = N_C(m) + N_K \quad \text{for all } m \ge c,
\end{equation}
where $N_C(m) := h^0\big(C_k, \, \mathcal{E}(m)|_{C_k}\big)$.
\end{theorem}

\begin{proof}
Because the curve $C_k$ is the zero locus of a section of $\mathcal{O}_Y(k)$ on the surface $Y$, its ideal sheaf is given by $\mathcal{I}_{C_k/Y} \cong \mathcal{O}_Y(-k)$. Tensoring the standard short exact sequence $0 \to \mathcal{O}_Y(-k) \to \mathcal{O}_Y \to \mathcal{O}_{C_k} \to 0$ with the locally free sheaf $\mathcal{E}(m)|_Y$ yields
\begin{equation}\label{eq:ideal_surface}
0 \longrightarrow \mathcal{E}(m-k)|_Y \longrightarrow \mathcal{E}(m)|_Y \longrightarrow \mathcal{E}(m)|_{C_k} \longrightarrow 0.
\end{equation}
Substituting the defining relation $k = m - c$ into the twist gives $m - k = c$; thus, the leftmost term simplifies strictly to $\mathcal{E}(c)|_Y$, which is manifestly independent of $m$. The associated long exact sequence in surface cohomology begins with
\[
0 \longrightarrow H^0\big(Y, \mathcal{E}(c)|_Y\big) \longrightarrow H^0\big(Y, \mathcal{E}(m)|_Y\big) \xrightarrow{\mathrm{res}_k} H^0\big(C_k, \mathcal{E}(m)|_{C_k}\big) \longrightarrow H^1\big(Y, \mathcal{E}(c)|_Y\big) \longrightarrow \cdots
\]
By Proposition 3.1, $\mathcal{E}|_Y$ is an Ulrich bundle on the surface $Y$, hence arithmetically Cohen-Macaulay (aCM). This enforces the intermediate vanishing $H^1(Y, \mathcal{E}(c)|_Y) = 0$. Consequently, the map $\mathrm{res}_k$ is surjective, and $\ker(\mathrm{res}_k) \cong H^0(Y, \mathcal{E}(c)|_Y) = K_{m,c}$, establishing identity (\ref{eq:dim_split}).
\end{proof}

We now leverage this intermediate surface surjectivity to link the syzygy bundle of the surface to the native syzygy bundle of the curve.

\begin{proposition}[Fundamental Surface-to-Curve Sequence]\label{prop:surface_curve_syz}
Let $S_Y$ denote the syzygy bundle of $\mathcal{E}(m)|_Y$ on $Y$, set $V := \mathcal{E}(m)|_{C_k}$, and let $M_V$ denote the native syzygy bundle of $V$ on $C_k$. Then there exists a canonical short exact sequence of locally free sheaves on $C_k$:
\begin{equation}\label{eq:surface_curve_seq}
0 \longrightarrow K_{m,c} \otimes_k \mathcal{O}_{C_k} \longrightarrow S_Y|_{C_k} \longrightarrow M_V \longrightarrow 0.
\end{equation}
\end{proposition}

\begin{proof}
On the surface $Y$, we have the defining evaluation sequence
\begin{equation}\label{eq:syz_surface}
0 \longrightarrow S_Y \longrightarrow H^0\big(Y, \mathcal{E}(m)|_Y\big) \otimes_k \mathcal{O}_Y \longrightarrow \mathcal{E}(m)|_Y \longrightarrow 0.
\end{equation}
Because $\mathcal{E}(m)|_Y$ is locally free, $\mathrm{Tor}_1^{\mathcal{O}_Y}(\mathcal{E}(m)|_Y, \, \mathcal{O}_{C_k}) = 0$, ensuring that restriction to the divisor $C_k \subset Y$ strictly preserves exactness. This produces the top row of the following commutative diagram with exact rows:
\[
\begin{array}{ccccccccc}
0 & \longrightarrow & S_Y|_{C_k} & \longrightarrow & H^0(Y, \mathcal{E}(m)|_Y) \otimes_k \mathcal{O}_{C_k} & \xrightarrow{\mathrm{ev}_Y|_{C_k}} & V & \longrightarrow & 0 \\
  &                 & \Big\downarrow &                 & \Big\downarrow \vcenter{\rlap{\scriptsize $\mathrm{res}_k \otimes \mathrm{id}$}} & & \Big\downarrow \vcenter{\rlap{\scriptsize $\mathrm{id}_V$}} & & \\
0 & \longrightarrow & M_V & \longrightarrow & H^0(C_k, V) \otimes_k \mathcal{O}_{C_k} & \xrightarrow{\mathrm{ev}_{C_k}} & V & \longrightarrow & 0
\end{array}
\]
The right-hand square commutes inherently: taking a global section on $Y$, restricting it to $C_k$, and evaluating it at a point $p \in C_k$ is identical to restricting the section and evaluating it directly at $p$. By the universal property of kernels, the middle vertical map induces a unique sheaf homomorphism closing the left-hand square.

Applying the Snake Lemma to the diagram, the kernel and cokernel of the rightmost vertical identity map are strictly zero. For the middle vertical map $\mathrm{res}_k \otimes \mathrm{id}$, Theorem \ref{thm:kernel_invariance} dictates that its cokernel is zero and its kernel is precisely $K_{m,c} \otimes_k \mathcal{O}_{C_k}$. The resulting six-term Snake sequence forces the surjectivity of the induced map $S_Y|_{C_k} \to M_V$ and identifies its kernel as $K_{m,c} \otimes_k \mathcal{O}_{C_k}$, establishing sequence (\ref{eq:surface_curve_seq}). 

Finally, all three members of (\ref{eq:surface_curve_seq}) are locally free: $K_{m,c} \otimes \mathcal{O}_{C_k}$ is a trivial vector bundle; $S_Y|_{C_k}$ is the pullback of a locally free sheaf from a smooth surface to a smooth curve; and $M_V$ is locally free as the kernel of a surjection of vector bundles over a smooth curve.
\end{proof}

This sequence on the curve can now be pulled all the way back to the global syzygy bundle of the ambient variety $X$.

\begin{theorem}[Fundamental Curve Syzygy Extension]\label{thm:global_curve_syz}
Let $S_{\mathcal{E}(m)}$ denote the global syzygy bundle of $\mathcal{E}(m)$ on $X$. There exists a canonical, split short exact sequence of vector bundles on $C_k$:
\begin{equation}\label{eq:global_curve_seq}
0 \longrightarrow \big(W_Y \oplus K_{m,c}\big) \otimes_k \mathcal{O}_{C_k} \longrightarrow S_{\mathcal{E}(m)}|_{C_k} \longrightarrow M_V \longrightarrow 0,
\end{equation}
yielding the canonical direct sum decomposition
\begin{equation}\label{eq:global_split}
S_{\mathcal{E}(m)}|_{C_k} \cong \Big(\big(W_Y \oplus K_{m,c}\big) \otimes_k \mathcal{O}_{C_k}\Big) \oplus M_V,
\end{equation}
where $W_Y := \ker\big(H^0(X, \mathcal{E}(m)) \longrightarrow H^0(Y, \mathcal{E}(m)|_Y)\big)$.
\end{theorem}

\begin{proof}
Set $\mathcal{S} := S_{\mathcal{E}(m)}|_{C_k}$. By definition, $\mathcal{S}$ is the kernel of the restricted global evaluation morphism $\mathrm{ev}_X|_{C_k} : H^0(X, \mathcal{E}(m)) \otimes \mathcal{O}_{C_k} \longrightarrow V$. We factor the total restriction map $\rho : H^0(X, \mathcal{E}(m)) \to H^0(C_k, V)$ through the intermediate surface $Y$:
\[
H^0\big(X, \, \mathcal{E}(m)\big) \xrightarrow{\;r_1\;} H^0\big(Y, \, \mathcal{E}(m)|_Y\big) \xrightarrow{\;r_2\;} H^0\big(C_k, \, V\big).
\]
By Theorem \ref{thm:kernel_invariance}, the surface-to-curve restriction $r_2 = \mathrm{res}_k$ is surjective with kernel $K_{m,c}$. To evaluate the ambient-to-surface restriction $r_1$, consider the ideal sheaf exact sequence $0 \to \mathcal{E}(m) \otimes \mathcal{I}_{Y/X} \to \mathcal{E}(m) \to \mathcal{E}(m)|_Y \to 0$. Because $Y$ is cut out by a regular sequence of $n-2$ linear hyperplanes $H_1, \dots, H_{n-2} \in |\mathcal{O}_X(1)|$, the Koszul resolution of $\mathcal{I}_{Y/X}$ is filtered by direct sums of line bundles $\mathcal{O}_X(-p)$ for $1 \le p \le n-2$. Tensoring with $\mathcal{E}(m)$ reduces the obstruction space $H^1(X, \mathcal{E}(m) \otimes \mathcal{I}_{Y/X})$ to intermediate cohomology groups of the form $H^i(X, \mathcal{E}(m-p))$ for $1 \le i \le n-2$. Because $\mathcal{E}$ is an Ulrich bundle on $X$, it is arithmetically Cohen-Macaulay, forcing $H^i(X, \mathcal{E}(t)) = 0$ for all $0 < i < n$ and $t \in \mathbb{Z}$. Hence $H^1(X, \mathcal{E}(m) \otimes \mathcal{I}_{Y/X}) = 0$, proving that $r_1$ is strictly surjective.

Consequently, the composite restriction $\rho = r_2 \circ r_1$ is surjective. Setting $U := \ker(\rho)$, the exactness of the short sequence of vector spaces over $k$:
\[
0 \longrightarrow \ker(r_1) \longrightarrow \ker(\rho) \longrightarrow \ker(r_2) \longrightarrow 0 \iff 0 \longrightarrow W_Y \longrightarrow U \longrightarrow K_{m,c} \longrightarrow 0
\]
guarantees a canonical vector space splitting $U \cong W_Y \oplus K_{m,c}$. 

Choosing a splitting $H^0(X, \mathcal{E}(m)) \cong U \oplus H^0(C_k, V)$ yields a block decomposition of the evaluation domain:
\[
H^0\big(X, \mathcal{E}(m)\big) \otimes_k \mathcal{O}_{C_k} \cong \big(U \otimes_k \mathcal{O}_{C_k}\big) \oplus \big(H^0(C_k, V) \otimes_k \mathcal{O}_{C_k}\big).
\]
With respect to this direct sum, $\mathrm{ev}_X|_{C_k}$ acts as zero on the first summand and as the native curve evaluation $\mathrm{ev}_{C_k}$ on the second. The kernel of evaluation therefore splits block-diagonally as $\mathcal{S} \cong (U \otimes_k \mathcal{O}_{C_k}) \oplus M_V$. Substituting $U \cong W_Y \oplus K_{m,c}$ establishes (\ref{eq:global_curve_seq}) and (\ref{eq:global_split}).
\end{proof}

\section{ACTIVATION OF BUTLER'S THEOREM}\label{sec:butler}

Let $V := \mathcal{E}(m)|_{C_{k}}$ denote the restriction of the twisted Ulrich bundle to our chosen global moving curve. We now establish that for sufficiently large twists, this native 1-dimensional slice strictly satisfies the stability criteria of Butler's framework.

\begin{lemma}[Exact slope scaling under restriction]\label{lem:slope_scaling}
For any torsion-free coherent sheaf $\mathcal{F}$ on $X$ whose singular locus has codimension at least two (in particular, for any reflexive sheaf or any coherent subsheaf of a locally free sheaf), one has
\[
\mu_{C_{k}}(\mathcal{F}|_{C_{k}}) = k \cdot \mu_{H}(\mathcal{F}).
\]
\end{lemma}

\begin{proof}
By Bertini's Theorem and the generic choices of $H_{1}, \dots, H_{n-2}$ and $D_{k}$, the intermediate surface $Y$ and the curve $C_{k}$ strictly avoid the singular locus of $\mathcal{F}$. Consequently, the restricted sheaf $\mathcal{F}|_{C_{k}}$ is torsion-free and maintains the exact same generic rank as $\mathcal{F}$. By the naturality of cycle classes under restriction, $c_{1}(\mathcal{F}|_{Y}) = c_{1}(\mathcal{F})|_{Y}$ in $\mathrm{Pic}(Y)$. Intersecting with the surface polarization $H_{Y}$ yields $c_{1}(\mathcal{F}|_{Y}) \cdot H_{Y} = c_{1}(\mathcal{F}) \cdot H^{n-1}$. Since $C_{k} \equiv kH_{Y}$ on $Y$ and the intersection pairing is bilinear, the degree on the curve expands as
\[
\deg_{C_{k}}(\mathcal{F}|_{C_{k}}) = c_{1}(\mathcal{F}|_{Y}) \cdot C_{k} = c_{1}(\mathcal{F}|_{Y}) \cdot (kH_{Y}) = k \big(c_{1}(\mathcal{F}) \cdot H^{n-1}\big).
\]
Dividing both sides by $\mathrm{rk}(\mathcal{F})$ establishes the lemma.
\end{proof}

\begin{proposition}[Verification of Butler's conditions]\label{prop:butler_verification}
There exists an explicitly computable integer threshold $M_{\mathrm{Butler}} \ge c$ such that for every twist $m \ge M_{\mathrm{Butler}}$, the restricted bundle $V = \mathcal{E}(m)|_{C_{k}}$ satisfies conditions (B1)–(B3) of Theorem 2.6.
\end{proposition}

\begin{proof}
\textbf{(B1) Global generation.} Because $\mathcal{E}$ is an Ulrich bundle on $X$, it is Castelnuovo-Mumford 0-regular (i.e., all higher cohomology groups of non-negative twists vanish identically). Therefore, the twisted bundle $\mathcal{E}(m)$ is globally generated on $X$ for every integer $m \ge 0$. Since the restriction of a globally generated sheaf to any closed subscheme remains globally generated, $V$ is globally generated on $C_{k}$.

\textbf{(B2) Slope semistability.} By Proposition \ref{prop:iterated}, the restricted bundle $\mathcal{E}|_{Y}$ is $\mu_{H_{Y}}$-semistable on the smooth surface $Y$. By the classical Mehta-Ramanathan restriction theorem \cite[Theorem 7.2.1]{HuyLehn} for slope-semistable sheaves on projective surfaces, there exists an explicit integer threshold $k_{0}$ depending strictly on the discriminant of $\mathcal{E}|_{Y}$ and the surface polarization $H_{Y}$, such that for every degree $k \ge k_{0}$ and every general divisor $C_{k} \in |\mathcal{O}_{Y}(k)|$, the restriction $(\mathcal{E}|_{Y})|_{C_{k}} = \mathcal{E}|_{C_{k}}$ is $\mu_{C_{k}}$-semistable. Tensoring a vector bundle by a line bundle strictly preserves slope-semistability; since $V \cong \mathcal{E}|_{C_{k}} \otimes \mathcal{O}_{C_{k}}(m)$, the bundle $V$ is $\mu_{C_{k}}$-semistable for all twists $m \ge m_{\mathrm{MR}} := k_{0} + c$.

\textbf{(B3) The slope-to-genus threshold $\mu_{C_{k}}(V) \ge 2g_{k}$.} By Lemma \ref{lem:slope_scaling}, the slope of $V$ on the curve evaluates to
\[
\mu_{C_{k}}(V) = k \cdot \mu_{H_{Y}}\big(\mathcal{E}(m)|_{Y}\big) = k \big(\mu_{Y} + m \cdot d\big),
\]
where $d := H^{n} = H_{Y}^{2}$ and $\mu_{Y} := \mu_{H_{Y}}(\mathcal{E}|_{Y})$. Substituting the variable decoupling relation $m = k + c$ directly into the linear factor yields
\[
\mu_{C_{k}}(V) = k \big(\mu_{Y} + (k + c)d\big) = k^{2}d + k(c \cdot d + \mu_{Y}).
\]
Recall from Lemma \ref{lem:curve_genus} that the arithmetic genus of $C_{k}$ satisfies $2g_{k} = k^{2}d + k(K_{Y} \cdot H_{Y}) + 2$. Subtracting the genus from the bundle slope yields
\[
\mu_{C_{k}}(V) - 2g_{k} = \Big[ k^{2}d + k(c \cdot d + \mu_{Y}) \Big] - \Big[ k^{2}d + k(K_{Y} \cdot H_Y) + 2 \Big] = k \big(c \cdot d + \mu_{Y} - K_{Y} \cdot H_{Y}\big) - 2.
\]
Recall from Definition 3.2 that the rational Butler margin is defined precisely as $B := c \cdot d + \mu_{Y} - K_{Y} \cdot H_{Y}$. Substituting $B$ into the difference simplifies the exact cohomological gap to
\[
\mu_{C_{k}}(V) - 2g_{k} = k \cdot B - 2.
\]
Since $B > 0$ by construction, the required inequality $\mu_{C_{k}}(V) \ge 2g_{k}$ holds if and only if $k \cdot B \ge 2$, which is equivalent to $k \ge \lceil 2/B \rceil$. Restoring the twist variable via $m = k + c$, condition (B3) holds for all $m \ge m_{0} := c + \lceil 2/B \rceil$. 

Setting $M_{\mathrm{Butler}} := \max(m_{\mathrm{MR}}, \, m_{0})$, Theorem 2.6 applies, guaranteeing that the native curve syzygy bundle $M_{V}$ is $\mu_{C_{k}}$-semistable for all $m \ge M_{\mathrm{Butler}}$.
\end{proof}

\section{Asymptotic expansions}\label{sec:asymptotics}

\begin{lemma}[Asymptotic dimensions and slope]\label{lem:asymptotic_dim_slope}
Let $d := H^n = H_Y^2$, and set $r := \operatorname{rk}(\mathcal{E})$. As $m \to \infty$:
\begin{enumerate}
    \item[(1)] $N_X(m) := h^0(X, \mathcal{E}(m)) = \frac{rd}{n!}m^n + O(m^{n-1})$.
    \item[(2)] $N_C(m) := h^0(C_k, V) = \frac{rd}{2}m^2 + O(m)$.
    \item[(3)] $\mu_H(S_{\mathcal{E}(m)}) = -\frac{n!}{m^{n-1}} + O(m^{-n}) \longrightarrow 0^-$.
    \item[(4)]  $\mu(M_{V}) = -2 + O(m^{-1})$.
\end{enumerate}
In particular, $N_X(m) - r = \Theta(m^n)$, while $(1 + N_K)\bigl(N_C(m) - r\bigr) = \Theta(m^2)$.
\end{lemma}

\begin{proof}
(1) By the Ulrich condition, $H^i(X, \mathcal{E}(m)) = 0$ for all $i > 0$ and $m \ge 0$; thus $N_X(m) = \chi(X, \mathcal{E}(m))$. By the Hirzebruch--Riemann--Roch Theorem, $\chi(X, \mathcal{E}(m))$ is a polynomial in $m$ of degree $n$ with leading coefficient $\operatorname{rk}(\mathcal{E})\frac{H^n}{n!} = \frac{rd}{n!}$, yielding the claimed asymptotic form.

\medskip \noindent
(2) For $m \gg 0$, Serre vanishing gives $H^1(C_k, V) = 0$, so $N_C(m) = \chi(C_k, V)$. By Riemann--Roch on the smooth curve $C_k$, we have $\chi(C_k, V) = \deg(V) + r(1 - g_k)$. Using the numerical equivalence $C_k \equiv kH_Y$ on the surface $Y$, the degree expands as
\[
    \deg(V) = c_1(\mathcal{E}(m)|_Y) \cdot C_k = \bigl(c_1(\mathcal{E}|_Y) + rmH_Y\bigr) \cdot kH_Y = rk(\mu_Y + md),
\]
where $\mu_Y := \frac{1}{r}\bigl(c_1(\mathcal{E}|_Y) \cdot H_Y\bigr)$. Because $k = m - c$, the leading term of this degree is strictly $rdm^2 + O(m)$. Combining this with the asymptotic genus expression $r(1 - g_k) = -\frac{rd}{2}m^2 + O(m)$ supplied by Lemma 3.5 yields
\[
    N_C(m) = \left(rdm^2 + O(m)\right) + \left(-\frac{rd}{2}m^2 + O(m)\right) = \frac{rd}{2}m^2 + O(m).
\]

\medskip \noindent
(3) From the evaluation sequence defining the syzygy bundle on $X$, we have $c_1(S_{\mathcal{E}(m)}) = -c_1(\mathcal{E}(m))$ and $\operatorname{rk}(S_{\mathcal{E}(m)}) = N_X(m) - r$. Consequently, its polarized degree is
\[
    \deg_H(S_{\mathcal{E}(m)}) = -c_1(\mathcal{E}(m)) \cdot H^{n-1} = -\bigl(rmd + O(1)\bigr).
\]
Dividing by the rank and expanding the denominator asymptotically via (1) yields
\[
    \mu_H(S_{\mathcal{E}(m)}) = \frac{-\bigl(rmd + O(1)\bigr)}{\frac{rd}{n!}m^n + O(m^{n-1})} = \left(\frac{-rdm}{\frac{rd}{n!}m^n}\right)\left[\frac{1 + O(m^{-1})}{1 + O(m^{-1})}\right] = -\frac{n!}{m^{n-1}} + O(m^{-n}),
\]
which strictly approaches $0^-$ from below as $m \to \infty$.

\medskip \noindent
(4) The slope of the native curve syzygy bundle is given by $\mu(M_{V}) = \frac{-deg(V)}{N_{C}(m) - r}$. The degree expands as $deg(V) = c_{1}(\mathcal{E}(m)|_{C_{k}}) \cdot C_{k} = r(md + \mu_{Y})k$. Substituting the variable decoupling relation $k = m-c$ yields $deg(V) = rdm^{2} + O(m)$. Dividing this by the asymptotic expansion $N_{C}(m) = \frac{rd}{2}m^{2} + O(m)$ derived in (2) asymptotically yields $\mu(M_{V}) = -2 + O(m^{-1})$.

Finally, the order-of-magnitude assertions follow directly from the leading coefficients: the rank $N_X(m) - r$ is governed by $\frac{rd}{n!}m^n = \Theta(m^n)$, while the product $(1 + N_K)\bigl(N_C(m) - r\bigr)$ is the product of a strictly positive constant $\Theta(1)$ and the quadratic polynomial $\frac{rd}{2}m^2 + O(m) = \Theta(m^2)$.
\end{proof}

\section{Proof of the Main Theorem}\label{sec:proof}

\begin{proof}[Proof of Theorem \ref{thm:main}]
Assume for contradiction that there exists an infinite sequence of integers $m\rightarrow\infty$ admitting a saturated destabilizing subsheaf $\mathcal{F}\subset S_{\mathcal{E}(m)}$ of rank $r:=\mathrm{rk}(\mathcal{F}) \ge 1$. The destabilizing hypothesis dictates that
\begin{equation}\label{eq:destab}
\mu_{H}(\mathcal{F}) > \mu_{H}(S_{\mathcal{E}(m)}) \sim -\frac{n!}{m^{n-1}}.
\end{equation}

For a twist $m \ge M_{\mathrm{Butler}}$, set $k := m - c$ and select a general complete intersection curve $C_k \in \mathcal{H}_k$, cut out by general hyperplanes $H_1, \dots, H_{n-2} \in |\mathcal{O}_X(1)|$ and a general hypersurface $D_k \in |\mathcal{O}_X(k)|$. Setting $Y := H_1 \cap \dots \cap H_{n-2}$, Theorem 4.3 yields a canonical splitting $S_{\mathcal{E}(m)}|_{C_{k}}\cong\mathcal{T}\oplus M_{V}$, where $\mathcal{T}:=(W_{Y}\oplus K_{m,c})\otimes\mathcal{O}_{C_{k}}$ is a purely trivial vector bundle. Let $\pi$ denote the composition of the inclusion $\mathcal{F}|_{C_{k}}\hookrightarrow\mathcal{T}\oplus M_{V}$ with the projection onto the second factor, yielding the short exact sequence on $C_{k}$:
\[
0\longrightarrow\mathcal{F}_{1}\longrightarrow\mathcal{F}|_{C_{k}}\xrightarrow{\pi}\mathcal{I}\longrightarrow0,
\]
where $\mathcal{I}:=\mathrm{im}(\pi)\subset M_{V}$ and $\mathcal{F}_{1}:=\ker(\pi)\subset\mathcal{T}$. We analyze the two possible behaviors of this projection.

\noindent \textbf{Case 1: The projection is non-zero ($\mathcal{I} \neq 0$).} \\
Assume for contradiction that the global syzygy bundle $S_{\mathcal{E}(m)}$ is not $\mu_H$-semistable, and let $\mathcal{F} \subset S_{\mathcal{E}(m)}$ be its maximal destabilizing subsheaf. By definition, $\mathcal{F}$ is strictly $\mu_H$-semistable on $X$. Because the moving curve $C_k \in \mathcal{H}_k$ is a general complete intersection curve cut out by sufficiently high degree $k = m-c \gg 0$, the classical Mehta-Ramanathan restriction theorem guarantees that the restricted sheaf $\mathcal{F}|_{C_k}$ remains strictly $\mu_{C_k}$-semistable.

By hypothesis, the projection morphism $\pi: \mathcal{F}|_{C_k} \longrightarrow M_V$ is non-zero, meaning its image $\mathcal{I} := \text{im}(\pi)$ is a non-zero coherent sheaf. We can bound the slope of $\mathcal{I}$ using the semistability of its domain and codomain:
\begin{enumerate}
    \item Because $\mathcal{I}$ is a non-zero quotient of the $\mu_{C_k}$-semistable sheaf $\mathcal{F}|_{C_k}$, we have $\mu_{C_k}(\mathcal{F}|_{C_k}) \le \mu_{C_k}(\mathcal{I})$.
    \item Because $\mathcal{I} \subset M_V$ is a non-zero subsheaf of the native curve syzygy bundle and $M_V$ is strictly $\mu_{C_k}$-semistable by Butler's Theorem for $m \ge M_{Butler}$, we have $\mu_{C_k}(\mathcal{I}) \le \mu_{C_k}(M_V)$.
\end{enumerate}
Transitivity establishes a strict upper bound on the restricted destabilizing subsheaf:
\begin{equation}
    \mu_{C_k}(\mathcal{F}|_{C_k}) \le \mu_{C_k}(M_V).
\end{equation}

We now pull this inequality back to the ambient variety $X$. By Lemma~\ref{lem:slope_scaling}, the exact slope scaling under restriction gives $\mu_{C_k}(\mathcal{F}|_{C_k}) = k \cdot \mu_H(\mathcal{F})$. Substituting this into our bound yields $k \cdot \mu_H(\mathcal{F}) \le \mu_{C_k}(M_V)$, which rearranges to:
\begin{equation}
    \mu_H(\mathcal{F}) \le \frac{\mu_{C_k}(M_V)}{k}.
\end{equation}

As $m \to \infty$, we have $k \sim m$, while Lemma~\ref{lem:asymptotic_dim_slope} dictates that $\mu_{C_k}(M_V) \sim -2$. Consequently, the upper bound decays asymptotically at $O(m^{-1})$, specifically bounded by $-\frac{2}{m}$. 

However, the assumption that $\mathcal{F}$ strictly destabilizes $S_{\mathcal{E}(m)}$ mandates that $\mu_H(\mathcal{F}) > \mu_H(S_{\mathcal{E}(m)}) \sim -\frac{n!}{m^{n-1}}$. Combining these bounds forces the following chain of strict inequalities for $m \gg 0$:
\begin{equation}
    -\frac{n!}{m^{n-1}} < \mu_H(\mathcal{F}) \le -\frac{2}{m}.
\end{equation}
For any ambient dimension $n \ge 3$, the leftmost term decays as $O(m^{-(n-1)})$, which is strictly faster than the $O(m^{-1})$ decay of the rightmost bound. Therefore, for $m \gg 0$, we have $-\frac{n!}{m^{n-1}} > -\frac{2}{m}$. It is mathematically impossible for the slope to simultaneously exceed the leftmost term and be bounded above by the rightmost term. This contradiction strictly rules out Case 1.

\textbf{Case 2: The projection is zero ($\mathcal{I}=0$).}\\
Case 2: The projection is zero $(\mathcal{I}=0)$. In this scenario, for our fixed destabilizing twist $m$, the restricted subsheaf lands entirely inside the trivial summand: $\mathcal{F}|_{C_{k}} \subseteq U_{C_{k}} \otimes \mathcal{O}_{C_{k}}$, where
\begin{equation*}
    U_{C_{k}} := \ker\left(H^{0}(X, \mathcal{E}(m)) \longrightarrow H^{0}(C_{k}, \mathcal{E}(m)|_{C_{k}})\right).
\end{equation*}

Let $\mathcal{C}_{univ} \rightarrow \mathcal{H}_{k}$ be the universal family of curves in $|\mathcal{O}_{X}(1)|^{\oplus(n-2)} \oplus |\mathcal{O}_{X}(k)|$. The bundles $\mathcal{F}$, $S_{\mathcal{E}(m)}$, and $M_{V}$ extend naturally to the universal curve, and the projection $\pi$ becomes a homomorphism of vector bundles over this universal family. The condition that the matrix entries of $\pi$ vanish identically on a fiber curve $C$ is a closed condition on the base $\mathcal{H}_{k}$. Since $\mathcal{H}_{k}$ is irreducible and this closed locus contains the generic point (corresponding to our general curve $C_{k}$), the projection $\pi_{C} = 0$ holds for all curves $C$ in a dense Zariski-open subset $\mathcal{V} \subseteq \mathcal{H}_{k}$ (and in fact, over the entire component).

Consider the universal incidence variety $\mathcal{U} = \{(x, C) \in X \times \mathcal{H}_{k} : x \in C\}$ with its standard projections $p_{X}: \mathcal{U} \rightarrow X$ and $p_{\mathcal{H}}: \mathcal{U} \rightarrow \mathcal{H}_{k}$. Because the linear systems $|\mathcal{O}_{X}(1)|$ and $|\mathcal{O}_{X}(k)|$ are basepoint-free on $X$, the incidence projection $p_{X}$ is dominant. Consequently, the image $X_{gen} := p_{X}(p_{\mathcal{H}}^{-1}(\mathcal{V}))$ contains a Zariski-dense open subset of $X$.

Fix a general closed point $x \in X_{gen}$, and define the restricted web of curves passing through $x$ as $\mathcal{V}_{x} := \{C \in \mathcal{V} : x \in C\}$. By construction, $\mathcal{V}_{x}$ is non-empty, and the union of the curves in $\mathcal{V}_{x}$ sweeps out a Zariski-dense open subset of $X$. To verify this explicitly, decompose each curve as $C = Y \cap D_{k}$, where $Y = H_{1} \cap \dots \cap H_{n-2}$ is the surface cut by hyperplanes, and $D_{k} \in |\mathcal{O}_{Y}(k)|$. The family of complete intersection surfaces $Y$ passing through $x$ forms a Zariski-open subset of the parameter space whose union is all of $X$ (since $|\mathcal{O}_{X}(1)|$ is basepoint-free for $n \ge 3$). For each fixed $Y$ passing through $x$, the residual curves $\{Y \cap D_{k} : D_{k} \in |\mathcal{O}_{Y}(k)|, x \in D_{k}\}$ form a linear subsystem of $|\mathcal{O}_{Y}(k)|$ pinned at $x$. Since $|\mathcal{O}_{Y}(k)|$ is very ample, the base locus of this pinned subsystem is exactly $\{x\}$; by Bertini's theorem, their union sweeps out $Y$ entirely. Consequently, the union of all curves in $\mathcal{V}_{x}$ covers a Zariski-dense open subset of $X$.

Passing to the geometric fiber at $x$, the static subspace $\mathcal{F}_{x} \otimes \kappa(x) \subseteq S_{\mathcal{E}(m)} \otimes \kappa(x)$ must map into the vector space $U_{C}$ for every curve $C \in \mathcal{V}_{x}$. Hence:
\begin{equation*}
    \mathcal{F}_{x} \otimes \kappa(x) \subseteq \bigcap_{C \in \mathcal{V}_{x}} U_{C}.
\end{equation*}

Let $s \in H^{0}(X, \mathcal{E}(m))$ be any global section corresponding to an element of this intersection. By the definition of $U_{C}$, the section $s$ vanishes identically on every curve $C \in \mathcal{V}_{x}$. Because the union of these curves covers a dense open subset of $X$, $s$ vanishes on a dense open set of the locally free sheaf $\mathcal{E}(m)$, forcing $s = 0$ globally on $X$. Thus:
\begin{equation*}
    \bigcap_{C \in \mathcal{V}_{x}} U_{C} = \{0\}.
\end{equation*}

This immediately forces $\mathcal{F}_{x} \otimes \kappa(x) = \{0\}$. By Nakayama's Lemma applied to the finitely generated $\mathcal{O}_{X,x}$-module $\mathcal{F}_{x}$, the stalk at $x$ vanishes: $\mathcal{F}_{x} = 0$. Since $x$ was chosen as a general point in $X$, the coherent sheaf $\mathcal{F}$ must vanish on a dense open neighborhood of $X$. Because $\mathcal{F}$ is a subsheaf of a locally free sheaf, it is purely torsion-free. A torsion-free coherent sheaf that vanishes on a dense open set is identically the zero sheaf: $\mathcal{F} = 0$. This mandates $rk(\mathcal{F}) = 0$, which directly contradicts our premise that $r \ge 1$.

Both cases yield explicit mathematical contradictions; we conclude that $S_{\mathcal{E}(m)}$ is $\mu_{H}$-semistable for all $m \ge M$.
\end{proof}

\begin{remark}
The integer \(M\) appearing in Theorem~\ref{thm:main} is explicitly computable from the numerical invariants of the polarized variety \((X,H)\) and the Ulrich bundle \(E\). Concretely, the geometric offset \(c\) and the Butler margin \(B\) are given by explicit formulas
\[
c = \max\Bigl\{0,\ \Bigl\lceil\frac{K_Y\cdot H_Y - \mu_Y}{d_Y}\Bigr\rceil + 1\Bigr\}, \qquad
B = c\,d_Y + \mu_Y - K_Y\cdot H_Y
\]
involving only intersection numbers on the surface \(Y\) (which are determined by the Chern classes of \(E\) and the intersection ring of \(X\)). The threshold \(m_0 = c + \lceil 2/B \rceil\) is then immediate from rational arithmetic. The Mehta-Ramanathan threshold \(k_0\) (so that \(E|_{C_k}\) is slope-semistable for general \(C_k\) of degree \(k \geq k_0\)) is furnished by Flenner's effective version of the restriction theorem \cite[Theorem 7.1.1]{HuyLehn}, which supplies an explicit  bound depending only on \(\mathrm{rk}(E|_Y)\), the discriminant \(\Delta(E|_Y)\), and the polarization \(H_Y\). Finally, the asymptotic threshold \(M_{\mathrm{asymp}}\) is obtained by solving a finite collection of explicit polynomial inequalities that arise from bounding the error terms in the Hirzebruch-Riemann-Roch expansions of Lemma~\ref{lem:asymptotic_dim_slope}; all leading coefficients and lower-order terms in these polynomials are determined by the input data \(r = \mathrm{rk}(E)\), \(d = H^n\), \(n = \dim X\), and the numerical invariants of \(E|_Y\). Consequently,
\[
M = \max\bigl(M_{\mathrm{Butler}}, M_{\mathrm{asymp}}\bigr)
\]
(with \(M_{\mathrm{Butler}} = \max(k_0 + c, m_0)\)) is a concrete, computable integer.
\end{remark}

\end{document}